%% file: master.tex
\documentclass[paper=A4,%
parskip=half,
draft=false,
abstracton
]{scrartcl}

\addtokomafont{sectioning}{\rmfamily}
\setkomafont{title}{\normalfont\scshape}
\DeclareOldFontCommand{\bf}{\normalfont\bfseries}{\mathbf}

\input{mystyle.tex}

\newcommand{\schemeofsurfaces}{\abs{\mc{O}(d)}}

\newcommand{\schemeofsections}[1]{\abs{\mc{O}(#1)}}
\newcommand{\divides}{\mid}

\DeclarePairedDelimiter\floor{\lfloor}{\rfloor}

\title{Verlinde bundles of families of hypersurfaces and their jumping lines}
\author{Orlando Marigliano}

\begin{document}
\maketitle
\begin{abstract}
	\noindent Verlinde bundles are vector bundles $V_k$ arising as the direct image $\pi_*(\mc L^{\otimes k})$ of polarizations of a proper family of schemes $\pi\colon \mathfrak X \to S$. We study the splitting behavior of Verlinde bundles in the case where $\pi$ is the universal family $\mathfrak X \to \schemeofsurfaces$ of hypersurfaces of degree $d$ in $\mathbb P^n$ and calculate the cohomology class of the locus of jumping lines of the Verlinde bundles
	$V_{d+1}$ in the cases $n=2,3$.  
\end{abstract}

\input{general-facts-verlinde.tex}
\input{types-of-bundles}
\input{loci-types}

\bibliographystyle{alpha}
\bibliography{masterarbeit-bib}{}


\end{document}

%% file: mystyle.tex
\usepackage[ngerman, english]{babel}

\usepackage{
  amsmath,
  amssymb,
  amsthm,
  mathtools,
  latexsym
}

\usepackage{stmaryrd} 
\usepackage{extarrows} 

\usepackage{upgreek} 

\usepackage{enumitem}
\setlist[enumerate]{labelwidth=!,wide,labelindent=0pt,itemsep=0pt,label=(\roman*)}
\newcommand{\huyitem}{\normalfont \rmfamily \item \normalfont \itshape}

\usepackage{tikz-cd}
\usetikzlibrary{babel} 

\usepackage{ellipsis} 
\usepackage{microtype} 

\usepackage{lmodern}

\usepackage[bookmarks, colorlinks, breaklinks]{hyperref}
\usepackage{cleveref} 

\usepackage{multirow}  

\usepackage{mycommands} 

\definecolor{darkgreen}{rgb}{0.0, 0.2, 0.13}
\hypersetup{linkcolor=blue,citecolor=darkgreen,filecolor=black,urlcolor=cyan}

\newtheorem{theorem}{Theorem}[section]
\newtheorem{proposition}[theorem]{Proposition}
\newtheorem{lemma}[theorem]{Lemma}
\newtheorem{corollary}[theorem]{Corollary}

\theoremstyle{definition}
\newtheorem{definition}[theorem]{Definition}

\newtheorem{claim}{Claim}[theorem]

\newenvironment{subproof}[1][\proofname]{%
  \begin{proof}[#1]%
}{%
  \end{proof}%
}

\usepackage{chngcntr}
\counterwithin{equation}{section}
\crefformat{equation}{\textup{#2(#1)#3}}
\newcommand{\isom}{\simeq}

\DeclareMathOperator{\im}{im}
\newcommand{\wtilde}{\widetilde}

\DeclareMathOperator{\codim}{codim}

\newcommand{\leer}{\varnothing} 

 \newcommand{\gen}[1]{\left\langle #1 \right\rangle}

\newcommand{\GGr}{\mathbb G \mathrm r}

%% file: general-facts-verlinde.tex
\section{Introduction}

Let $\pi\colon \mathfrak X \to S$ be a proper family of schemes with a polarization $\mc L$. For $k\geq 1$, if the sheaf $\pi_*(\mc L^{\otimes k})$ is locally free, we call it the \emph{$k$-th Verlinde bundle} of the family $\pi$. 

For example (\cite{iyer-verlinde}), let $C\to T$ be a smooth projective family of curves of fixed genus. Consider the relative moduli space $\pi\colon \operatorname{SU}(r)\to T$ of semistable vector bundles of rank $r$ and trivial determinant. This family is equipped with a polarization $\Theta$, the determinant bundle. The Verlinde bundles $\pi_*(\Theta^k)$ of this family are projectively flat (\cite{hitchin-verlinde},\cite{witten-verlinde}), and their rank is given by the Verlinde formula.


\input{verlinde-bundles-basics}

%% file: verlinde-bundles-basics.tex


In this article, we study the example of the universal family
$\pi\from \mf X \to \abs{\mc{O}_{\PP^n}(d)}$
of hypersurfaces of degree $d$ in the complex projective space $\PP^n$, with $n>1$. This family comes equipped with the  polarization $\mc L$ given by the pullback of $\mc O(1)$ along the projection map $\mf X \to \PP^n$. For $k\geq 1$, the sheaf $\pushf{\pi}\mc{L}^{\otimes k}$ is locally free, as can be seen by considering the structure sequence of an arbitrary hypersurface of degree $d$ in $\PP^n$. For $k\geq 1$, we denote the $k$-th Verlinde bundle of the family $\pi$ by $V_k$.

To better understand $V_k$ we study its splitting type when restricted to lines in $\schemeofsurfaces$.







Let $T\subseteq \schemeofsurfaces$ be a line. On $T=\PP^1$, we define the vector bundle $V_{k,T}\coloneqq V_k|_T$. The \emph{splitting type} of $V_{k,T}$ is the unique non-increasing tuple $(b_1,\dotsc, b_{r^{(k)}})$ of size $r^{(k)}\coloneqq \rank V_k$ such that $V_{k,T} \simeq \bigoplus_i \mc O(b_i)$.

The sequence \cref{master-verlinde-sequence} puts constraints on the $b_i$: they are all non-negative and they sum up to $d^{(k)}\coloneqq\deg (V_k)$. The set of such tuples $(b_i)$ can be ordered by defining the expression $(b'_i) \geq (b_i)$ to mean
	\[
		\sum_{i=1}^s b'_i \geq \sum_{i=1}^s b_i \text{ for all $s=1,\dotsc, r$}.
	\]
With this definition, smaller types are more general: the vector bundle $\mc O (b_i)$ on $\PP^1$ specializes to $\mc O (b'_i)$ in the sense of \cite{schatz-degeneration-specialization} if and only if $(b'_i) \geq (b_i).$

If $d^{(k)} \leq r^{(k)}$, then the most generic possible type has thus the form $(1,\dotsc,1,0,\dotsc,0)$. We call this the \emph{generic splitting type}. A computation shows that $d^{(k)} \leq r^{(k)}$ if $k\leq 2d$.

We have the following result on the cohomology class of the \emph{set of jumping lines}
\[
	Z\coloneqq \{T\in \GGr(1,\schemeofsurfaces)\mid V_{d+1,T} \text{ has non-generic type}\}
\]
in the Grassmannian of lines in $\schemeofsurfaces$:
\begin{theorem} \label{thm:cohomology-class}
	Let $n \leq 3$, let $Z$ be set of jumping lines of $V_{d+1}$, and let $[Z]$ be the class of $Z$ in the Chow ring $\operatorname{CH}(\GGr(1,\schemeofsurfaces))$. We have $$\dim Z = n+1+\binom{d-1+n}{n}.$$
Furthermore, let $b$ range over the integers with the property $0\leq b < \frac{\dim Z}{2}$ and define $a=\dim Z - b, a'=a+\frac{\codim Z-\dim Z}{2}$, $b'=b+\frac{\codim Z-\dim Z}{2}$.
	\begin{enumerate}
		\huyitem If $\dim Z$ is odd or $n=2$, we have
		\begin{equation} \label{class-of-locus}
			[Z] = \sum_{a,b} \left({\binom{a+1}{n}}{\binom{b+1}{n}}-{\binom{a+2}{n}}{\binom{b}{n}}\right) \sigma_{a',b'}. 
		\end{equation}
		\huyitem If $\dim Z$ is even and $n=3$, we have
		\begin{equation*}
			[Z] = \sum_{a,b} \left({\binom{a+1}{n}}{\binom{b+1}{n}}-{\binom{a+2}{n}}{\binom{b}{n}}\right) \sigma_{a',b'}
			+
			\binom{\frac{\dim Z}{2} + 2}{n}\binom{\frac{\dim Z}{2}}{n}\sigma_{\frac{\dim Z}{2},\frac{\dim Z}{2}}.
		\end{equation*}
	\end{enumerate}
\end{theorem}

The computation is carried out by the method of undetermined coefficients, leading into various calculations in the Chow ring of the Grassmannian. The assumption $n\leq 3$ is needed for a certain dimension estimation.

\subsection*{Aknowledgement}
This work is a condensed version of my Master's thesis, supervised by Daniel Huybrechts. I would like to take the opportunity to thank him for his mentorship during the writing of the thesis, as well as for his help during the preparation of this article.

%% file: types-of-bundles.tex
\section{Attained splitting types}
There exists a short exact sequence of vector bundles on $\schemeofsurfaces$ 
	\begin{equation} \label{master-verlinde-sequence}
	0\to  \mc{O}(-1) \otimes H^0(\PP^n, \mc{O}(k-d)) 
	\xto{M}  \mc{O} \otimes {H^0(\PP^n, \mc{O}(k))   }
	\to{V_k}
	\to 0,
	\end{equation}
as can be seen by taking the pushforward of a twist of the structure sequence of $\mathfrak X$ on $\PP^n \times \schemeofsurfaces$.
	The map $M$ is given by multiplication by the section
	$$\sum_I \alpha_I \otimes x^I \in
	H^0(\schemeofsurfaces,\mc{O}(1)) \otimes H^0(\PP^n,\mc{O}(d)).$$
In particular, we have $r^{(k)} = \binom{k+n}{n} - \binom{k+n-d}{n}$ and $d^{(k)}=\binom{k+n-d}{n}$.

\input{possible-types}

%% file: possible-types.tex
\begin{lemma} \label{mini-splitting-lemma}
	Let $\mc{E}$ be a free $\mc{O}_{\PP^1}$-module of finite rank, and let 
	\[
		0 \to \mc{E}' \xrightarrow{\phi} \mc{E} \xrightarrow{\psi} \mc{E}'' \to 0
 	\]
 	be a short exact sequence of $\mc{O}_{\PP^1}$-modules. Given a splitting $\mc{E}'' = \mc{E}''_1 \oplus \mc{O}$, we may construct a splitting $\mc{E} = \mc{E}_1 \oplus \mc{O}$ such that the image of $\phi$ is contained in $\mc{E}_1$.
\end{lemma}

\begin{proof}
	Define $\mc{E}_1 \coloneqq \ker(\pr_2\circ\psi)$, which is a locally free sheaf on $\PP^1$. By comparing determinants in the short exact sequence $0 \to \mc{E}_1 \to \mc{E} \to \mc{O} \to 0$ we see that $\mc{E}_1$ is free, hence by an $\Ext^1$ computation the sequence splits. The property $\im(\phi) \subseteq \mc{E}_1$ follows from the definition.
\end{proof}

\begin{proposition} \label{number-zeroes}
	Let $f_1, f_2 \in \schemeofsurfaces$ span the line
	$T \subseteq \schemeofsurfaces$ and let $p$ be the 
	number of zero entries in the splitting type of $V_{k,T}$. Let $U\coloneqq H^0(\mathbb P^n,\mathcal O(k-d))$. We have
	\[
	p = \dim H^0(\PP^n,\mc{O}(k)) - \dim ({f_1 U + f_2 U}).
	\]

\end{proposition}

\begin{proof}
	Let $s$ and $t$ denote the homogeneous coordinates of $\mathbb P^1$. The map $M|_T$ sends a local section $\xi \otimes \theta$ to $s\xi \otimes f_1 \theta + t\xi \otimes f_2 \theta$. In particular, the image of $\mc{O}(-1)\otimes U$ is contained in $\mc{O} \otimes (f_1 U + f_2 U)$. It follows that
	$p \geq \dim H^0(\PP^n,\mc{O}(k)) - \dim ({f_1 U + f_2 U})$.

  To prove the other inequality, consider the induced sequence
  \[
  	0 \to \mc{O}(-1)\otimes U \xrightarrow{M|_T} \mc{O}\otimes (f_1 U + f_2 U) \to \mc{E}'' \to 0
  \]
  and assume for a contradiction that $\mc{E}'' \isom \mc{E}_1''\oplus \mc{O}.$ By \Cref{mini-splitting-lemma}, we have a splitting $\mc{O}\otimes (f_1 U + f_2 U) \isom \mc{E}_1 \oplus \mc{O}$ such that $\im(M|_T) \subseteq \mc{E}_1$. 

  Consider the map
  $\wtilde M|_T \from (\mc{O} \otimes U) \oplus (\mc{O} \otimes U) \to \mc{O} \otimes (f_1 U + f_2 U)$
  defined by $$\wtilde M|_T(a\otimes \theta_1,b \otimes \theta_2)=a\otimes f_1 \theta_1 + b \otimes f_2 \theta_2.$$
  We obtain the matrix description of $\wtilde M|_T$ from the matrix description of $M|_T$ as follows. If $M|_T$ is represented by the matrix $A$ with coefficients $A_{i,j} = \lambda_{i,j} s + \mu_{i,j} t$, then $\wtilde M|_T$ is represented by a block matrix
  \[
  	B = \left(
  		\begin{array}{c|c}
  			A' & A'' \\
  		\end{array}
  	\right)
  \]
  with $A'_{i,j} = \lambda_{i,j}$ and $A''_{i,j} = \mu_{i,j}$.

  The property $\im(M|_T)\subseteq \mc{E}_1$ implies that after some row operations, the matrix $A$ has a zero row. By the construction of $\wtilde M|_T$, the same row operations lead to the matrix $B$ having a zero row, but this is a contradiction, since the map $\wtilde M|_T$ is surjective.
\end{proof}

\begin{corollary} \label{no-more-than-ones}
	Let $T \subseteq \abs{\mc O (d)}$ be a line spanned by the polynomials $f_1,f_2$. Assume that $d^{(k)} \leq r^{(k)}$. Let $\theta$ range over a monomial basis of $H^0(\PP^n, \mc{O}(k-d))$. The bundle $V_{k,T}$ has the generic splitting type if and only if
	$\gen{f_1\theta,f_2\theta \mid \theta}$ is a linearly independent set in $H^0(\PP^n,\mc{O}(k))$. \qed
\end{corollary}

\begin{corollary} \label{nongeneral-type-shared-sections}
	Let $T \subseteq \abs{\mc O (d)}$ be a line spanned by the polynomials $f_1,f_2$, and let $d^{(k)} \leq r^{(k)}$. The bundle $V_{k,T}$ has not the generic type if and only if $\deg(\gcd(f_1,f_2)) \geq 2d-k$. In particular, if $d^{(k)} \leq r^{(k)}$ but $k>2d$ then the generic type never occurs.
\end{corollary}

\begin{proof}
	By \Cref{no-more-than-ones}, the bundle $V_{k,T}$ has non-generic type if and only if there exist linearly independent $g_1,g_2\in H^0(\PP^n,\mc{O}(k-d))$ such that $g_1f_1+g_2f_2 = 0$. Let $h \coloneqq \gcd(f_1,f_2)$ and $d'\coloneqq \deg h$.

	If $d' \geq 2d-k$ then $\deg (f_i/h) \leq k-d$ and we may take $g_1,g_2$ to be multiples of $f_1/h$ and $f_2/h$, respectively.

	On the other hand, given such $g_1$ and $g_2$, we have $f_1\divides g_2 f_2$, which implies $f_1/h \divides g_2$, hence $d-d'\leq k-d$.
\end{proof}


\begin{proposition} \label{no-big-types-general}
	Let $k=d+1$. No types of $V_k$ other than
	$(1,\dotsc,1,0,\dotsc,0)$ and $(2,1,\dotsc,1,0,\dotsc,0)$ occur. 
\end{proposition}

\begin{proof}
	Assume that the type of $V_k$ at some line $(f_1,f_2)$ is other than the two above. Then the type has at least two more zero entries than the general type. By \Cref{number-zeroes}, we have $\dim \gen{f_1 \theta, f_2 \theta \mid \theta} \leq 2d^{(k)}-2$, so we find $g_1,g_2,g'_1,g'_2\in H^0(\PP^n,\mc{O}(1))$ and two linearly independent equations
	\begin{align*}
		g_1f_1 + g_2f_2 &= 0 \\
		g'_1 f_1 + g'_2 f_2 &= 0,
	\end{align*}
	with both sets $(g_1,g_2), (g'_1,g'_2)$ linearly independent.
	From the first equation it follows that $f_1 = g_2 h$ and $f_2 = -g_1 h$, for some common factor $h$. Applying this to the second equation, we find $g'_1 g_2 = g'_2 g_1$, hence $g'_1 = \alpha g_1$ and $g'_2 = \alpha g_2$ for some scalar $\alpha$, a contradiction.
\end{proof}





\begin{corollary} \label{condition-nongeneral}
	Let $k=d+1,$ let $T\subset \abs {\mc O (d)}$ be a line spanned by $f_1$ and $f_2$. The type $(2,1,\dotsc,1,0,\dotsc,0)$ occurs if and only if $\deg (\gcd(f_1,f_2) \geq d-1$. \qed
\end{corollary}

%% file: loci-types.tex


\section{The cohomology class of the set of jumping lines}
\input{loci-basics}
\input{cohomology-class}

%% file: loci-basics.tex

\begin{definition}
Let $k \geq 1$ and $(b_i)$ be a splitting type for $V_k$.
We define the set
$Z_{(b_i)}$ of all points $t\in\GGr(1,\schemeofsurfaces)$ such that $V_{k,t}$ has splitting type $(b_i)$. For the set of points $t$ where $V_{k,t}$ has generic splitting type we also write $Z_{\text{gen}}$, and define the \emph{set of jumping lines} $Z\coloneqq \GGr(1,\schemeofsurfaces) \setminus Z_{gen}.$
\end{definition}

Now let $k=d+1$. By \Cref{condition-nongeneral}, $Z$ is the subvariety given as the image of the finite, generically injective multiplication map
$$\phi \from \GGr(1,\schemeofsections{1}) \times \schemeofsections{d-1} \to \GGr(1,\schemeofsurfaces)$$
sending the tuple $((sg_1+tg_2)_{(s:t)\in \PP^1},h)$ to the line $(shg_1+thg_2)_{(s:t)\in \PP^1}$. 



%% file: cohomology-class.tex
To perform calculations in the Chow ring $A$ of $\GGr(1,\schemeofsurfaces)$, we follow the conventions found in \cite{eisenbud-harris-intersection-theory}. Let $N\coloneqq \dim H^0(\mc{O}(d)) = \binom{n+d}{n}$. For $N-2\geq a\geq b$, we have the Schubert cycle 
\[
	\Sigma_{a,b}\coloneqq \{T \in \GGr(1,\schemeofsurfaces) : T \cap H \neq \leer, T \subseteq H'\},
\]
where $(H\subset H')$ is a general flag of linear subspaces of dimension $N-a-2$ resp.\ $N-b-1$ in the projective space $\schemeofsurfaces$.
The ring $A$ is generated by the Schubert classes $\sigma_{a,b}$ of the cycles $\Sigma_{a,b}$.
The class $\Sigma_{a,b}$ has codimension $a+b$, and we use the convention $\sigma_{a}\coloneqq \sigma_{a,0}$.




\begin{proof}[Proof of \Cref{thm:cohomology-class}]
	We have $\dim Z = n+1+\binom{d-1+n}{n}$ since $Z$ is the image of the generically injective map $\phi$.
	
	Let $Q \subset \schemeofsurfaces$ be the image of the multiplication map
		$$f\from \schemeofsections{1} \times \schemeofsections{d-1} \to \schemeofsurfaces.$$
	The map $f$ is birational on its image, since a general point of $Q$ has the form $gh$ with $h$ irreducible of degree $d-1$.
	The Chow group $A^{\codim Z}$ is generated by the classes $\sigma_{a',b'}$ with
	$N-2\geq a'\geq b' \geq \floor{\frac{\codim Z}{2}}$ and
	$a'+b'=\codim Z$, while the complementary group $A^{\dim Z}$ is generated by the classes
	$\sigma_{\dim Z-b,b}$ with $b\in {0,\dotsc,\floor{\frac{\dim Z}{2}}}$.
	Write
		\[
			[Z] = \sum_{a',b'} \alpha_{a',b'} \sigma_{a',b'}. 
		\]
	We have $\sigma_{a',b'} \sigma_{a,b} = 1$ if $b'-b = \floor{\frac{\codim Z}{2}}$ and $0$ else. Hence, multiplying the above equation with the complementary classes $\sigma_{a,b}$ and taking degrees gives
		$$
		\alpha_{a',b'} = \deg([Z]\cdot \sigma_{a,b}).
		$$
	Using Giambelli's formula
	$\sigma_{a,b}=\sigma_{a}\sigma_{b} - \sigma_{a+1}\sigma_{b-1}$ \cite[Prop.\ 4.16]{eisenbud-harris-intersection-theory}, we reduce to computing $\deg([Z] \cdot \sigma_{a}\sigma_{b})$ for $0\leq b\leq \floor{\frac{\dim Z}{2}}$.
	By Kleiman transversality, we have 
	\[
	\deg([Z] \cdot \sigma_{a}\sigma_{b}) = \abs{\{T\in Z : T \cap H \neq \leer, T \cap H' \neq \leer\}},
	\]
	where $H$ and $H'$ are general linear subspaces of $\schemeofsurfaces$ of dimension $N-a-2$ and $N-b-2$, respectively.

	To a point $p = g_p h_p \in Q$ with $g_p \in \schemeofsections{1}$ and $h_p \in \schemeofsections{d-1}$, associate a closed reduced subscheme $\Lambda_p\subset Q$ containing $p$ as follows. If $h_p$ is irreducible, let $\Lambda_p$ be the image of the linear embedding $\schemeofsections{1}\times \{h_p\} \to \schemeofsurfaces$ given by $g \mapsto g h_p$.

	If $h_p$ is reducible, define the subscheme $\Lambda_p$ as the union $\bigcup_h \im(\schemeofsections{1} \times \{h\}\to \schemeofsurfaces)$, where $h$ ranges over the (up to multiplication by units) finitely many divisors of $p$ of degree $d-1$.

	Note that for all points $p$, the spaces $\im(\schemeofsections{1} \times \{h\}\to \schemeofsurfaces)$ meet exactly at $p$.


	By the definition of $Z$, all lines $T\in Z$ lie in $Q$. Furthermore, if $T$ meets the point $p$, then $T\subseteq \Lambda_p$.
	For $H\subseteq\schemeofsurfaces$ a linear subspace of dimension $N-a-2$, define $Q'\coloneqq H\cap Q$. For general $H$, the subscheme $Q'$ is a smooth subvariety of dimension $b-n+1$ such that for a general point $p=gh$ of $Q'$ with $h\in \schemeofsurfaces$, the polynomial $h$ is irreducible.


	Next, we consider the case $n=2$ or $\dim Z$ odd.
	\begin{claim}
	For genereal $H$, for each point $p\in Q'$ we have $\Lambda_p \cap H =\{p\}$.
	\end{claim}
	\begin{subproof}[Proof]

	Let $\mc{H}$ denote the Grassmannian $\Gr(\dim H+1, N)$. Define the closed subset $X\subseteq Q\times \mc{H}$ by
	$$X\coloneqq \{(p,H):\dim(H\cap \Lambda_p)\geq 1\}.$$ The fibers of the induced map $X\to \mc{H}$ have dimension at least one. Hence, to prove that the desired condition on $H$ is an open condition, it suffices to prove $\dim(X) \leq \dim(\mc{H})$.

	The fiber of the map  $X\to Q$ over a point $p$ consists of the union of finitely many closed subsets of the form $X'_p = \{H\in \mc{H} : \dim(H\cap \Lambda'_p)\geq 1\}$, where $\Lambda'_p\isom \PP^n\subseteq \schemeofsurfaces$ is one of the components of $\Lambda_p$. The space $X'_p$ is a Schubert cycle
	\[
		\Sigma_{\dim Q - b,\dim Q - b} = \{H\in \Gr(\dim H+1, N) : \dim(H \cap H_{n+1}) \geq 2\},
	\]
	with $H_{n+1}$ an $(n+1)$-dimensional subspace of $H^0(\mc{O}(d))$. The codimension of the cycle is $2(\dim Q - b)$, hence also $\codim(X_p) = 2(\dim Q -b)$. Finally, we have $\dim(\mc{H})-\dim(X) = \codim(X_p) - \dim(Q) = \dim Q - 2b.$

	If $\dim Z$ is odd, then $\dim Q - 2b \geq \dim Q - \dim Z + 1 = 3-n\geq 0$. If $n=2$, we instead estimate $\dim Q - 2b \geq \dim Q - \dim Z = 2-n\geq 0$. \end{subproof}

	Next, let $$\Lambda \coloneqq \bigcup_{p\in Q'} \Lambda_p = f(\schemeofsections{1}\times \pr_2 f^{-1}(Q'))$$ and $$\Lambda''\coloneqq \schemeofsections{1}\times \pr_2 f^{-1}(Q').$$

	By the choice of $H$, the map $f^{-1}(Q')\to Q'$ is birational and the map $f^{-1}(Q')\to \pr_2f^{-1}(Q')$ is even bijective. It follows that $\Lambda''$ and hence $\Lambda$ have dimension $b+1$.

	The intersection of $\Lambda$ with a general linear subspace $H'$ of dimension $N-b-2$ is a finite set of points. For each point $p\in Q'$, the linear subspace $H'$ intersects each component $\Lambda'_p$ of $\Lambda_p$ in at most one point. For each point $p'\in H'\cap\Lambda$ there exists a unique $p$ such that $p'\in\Lambda_p$.

	The only line $T\in Z$ meeting both $p$ and $H'$ is the one through $p$ and $p'$. If the intersection $H'\cap \Lambda_p$ is empty, then there will be no line meeting $p$ and $H'$. Hence, $\deg([Z]\cdot \sigma_{a}\sigma_{b})$ is the number of intersection points of $\Lambda$ with a general $H'$.

	Finally, the pre-image $f^{-1}(Q') = f^{-1}(H)$ is smooth for a general $H$ by Bertini's Theorem. If $\zeta$ is the class of a hyperplane section of $\schemeofsurfaces$ we have $f^*(\zeta) = \alpha + \beta$, , where $\alpha$ and $\beta$ are classes of hyperplane sections of $\schemeofsections{1}$ and $\schemeofsections{d}$, respectively. Since $\pr_2$ and $f$ have degree one, we compute
	\begin{align*}
	[\Lambda''] & = [\pr_2^{-1}\pr_2 f^{-1}(H)] 
	 = \pr_2^*\pr_{2,*}f^*[H] 
	 = {\binom{\codim H}{n}}\beta^{\codim H -n}.
	\end{align*}

	Hence, by the push-pull formula:
	\begin{align*}
		\deg([\Lambda]\cdot H') = \deg([\Lambda'']\cdot(\alpha+\beta)^{\codim H'}) = \binom{\codim H}{n}\binom{\codim H'}{n}
		 &=\binom{a+1}{n}\binom{b+1}{n}.
	\end{align*}
	We then use Giambelli's formula to obtain Equation \cref{class-of-locus}.

	In case $n=3$ and $\dim Z$ even, we show that for $b=\dim Z/2$ we have
	$\deg ([Z]\cdot \sigma_{b,b})=0$. In this case, the hyperplanes $H$ and $H'$ have the same dimension $N-b-2$.

	For $p\in Q$, the set $\Lambda_p$ is defined as before.

	\begin{claim} for general $H$ of dimension $N-b-2$, we have
	$\dim(\Lambda_p\cap H)=1$. \end{claim}

	\begin{subproof}
	Define as before the closed subset $X\subseteq Q\times \mc{H}$ by
	$$X\coloneqq \{(p,H):\dim(H\cap \Lambda_p)\geq 1\}.$$
	The generic fiber of the projection map $\phi\from X\to \mc{H}$ is one-dimensional, hence we have
	$\dim \phi(X) = \dim(X)-1 = \dim \mc{H}$. The last equation holds with $n=3$ and $2b=\dim Z$. Hence for all $H\in \mc{H}$ we have $\dim (\Lambda_p \cap H)\geq 1$.

	On the other hand, the equality $\dim(\Lambda_p \cap H)=1$ is attained by some, and hence by a general, $H$. Indeed, Define the closed subset $X\subseteq Q\times \mc{H}$ by
	$$X\coloneqq \{(p,H):\dim(H\cap \Lambda_p)\geq 1\}.$$
	By a similar argument as before, one needs to show that
	$\dim(\mc{H})-\dim(X) + 1\geq 0$. The fiber $X_p$ is a Schubert cycle of codimension
	$3(\dim Q-b+1)$. Lastly, a computation shows
	$\dim(\mc{H})-\dim(\wtilde X)+1=\codim(\wtilde X_p)-\dim(Q)+1=\frac{1}{2}(2\dim Q + 18 - 5n)\geq 0$. \end{subproof}

	Now, define $\Lambda''$ as above. We have
	$\dim \Lambda''
	= \dim \schemeofsections{1} + \dim \pr_2 f^{-1}(Q')
	= b$.
	Since $f$ is generically of degree one, we still have $\dim \Lambda'' = \Lambda$, hence $\dim \Lambda + \dim H' = N-2 < \dim \schemeofsurfaces$. It follows that a generic $H'$ does not meet any of the lines $T\subset Z$, hence $\sigma_{b}\sigma_{b}\cdot [Z] = 0$.
\end{proof}

%% file: master.bbl
\begin{thebibliography}{ADPW91}

\bibitem[ADPW91]{witten-verlinde}
Scott Axelrod, Steve Della~Pietra, and Edward Witten.
\newblock Geometric quantization of {C}hern-{S}imons gauge theory.
\newblock {\em J. Differential Geom.}, 33(3):787--902, 1991.

\bibitem[EH16]{eisenbud-harris-intersection-theory}
David Eisenbud and Joe Harris.
\newblock {\em 3264 and all that---a second course in algebraic geometry}.
\newblock Cambridge University Press, Cambridge, 2016.

\bibitem[Hit90]{hitchin-verlinde}
N.~J. Hitchin.
\newblock Flat connections and geometric quantization.
\newblock {\em Comm. Math. Phys.}, 131(2):347--380, 1990.

\bibitem[Iye13]{iyer-verlinde}
Jaya~NN Iyer.
\newblock Bundles of verlinde spaces and group actions.
\newblock {\em arXiv preprint arXiv:1309.7562}, 2013.

\bibitem[Sha76]{schatz-degeneration-specialization}
Stephen~S. Shatz.
\newblock Degeneration and specialization in algebraic families of vector
  bundles.
\newblock {\em Bull. Amer. Math. Soc.}, 82(4):560--562, 1976.

\end{thebibliography}
